\documentclass[12pt, a4paper]{article}
\usepackage[margin=2.5cm]{geometry}
\usepackage[utf8]{inputenc} 
\usepackage{lmodern}
\usepackage[T1]{fontenc}
\usepackage[english]{babel}

\usepackage[runin]{abstract}
\usepackage{titling}

\usepackage{amsmath}
\usepackage{amsthm}
\usepackage{amsfonts}
\usepackage{amssymb}
\usepackage{mathtools}
\usepackage{clrscode}
\usepackage{enumitem}
\usepackage{mdwlist}
\usepackage{color}

\usepackage{cite}

\abslabeldelim{.}

\newcommand{\keywordsname}{Key words}
\makeatletter
\newcommand{\keywords}[1]{%
\def\thekeywords{#1}%
\begin{@bstr@ctlist}
\hspace*{\abstitleskip}{\abstractnamefont\keywordsname\@bslabeldelim}\abstracttextfont\
#1%
\par\end{@bstr@ctlist}
}
\makeatother

\newcommand{\subjclassname}{Mathematics subject classification}
\makeatletter
\newcommand{\subjclass}[2][2020]{%
\begin{@bstr@ctlist}
\hspace*{\abstitleskip}{\abstractnamefont\subjclassname\ (#1)\@bslabeldelim}\abstracttextfont\
#2%
\par\end{@bstr@ctlist}
}
\makeatother

\makeatletter
\def\and{
	\end{tabular}%
	and%
	\begin{tabular}[t]{c}}%
\makeatother

\makeatletter
\def\thanks#1{
\protected@xdef\@thanks{\@thanks
\protect\footnotetext[\the\c@footnote]{#1}}%
}
\makeatother

\makeatletter
\let\addresses\@empty      
\newcommand{\address}[2][]{\g@addto@macro\addresses{\address{#1}{#2}}}
\newcommand{\curraddr}[2][]{\g@addto@macro\addresses{\curraddr{#1}{#2}}}
\newcommand{\email}[2][]{\g@addto@macro\addresses{\email{#1}{#2}}}
\newcommand{\urladdr}[2][]{\g@addto@macro\addresses{\urladdr{#1}{#2}}}
\def\enddoc@text{
  \ifx\@empty\addresses \else\@setaddresses\fi}
\AtEndDocument{\enddoc@text}
\def\emailaddrname{E-mail address}
\def\@setaddresses{\par
  \nobreak \begingroup
  \interlinepenalty\@M
  \def\address##1##2{\begingroup%
    \par\addvspace\bigskipamount
    \@ifnotempty{##1}{(\ignorespaces##1\unskip) }%
    {\noindent\ignorespaces##2}\par\endgroup}%
  \def\email##1##2{\begingroup
    \@ifnotempty{##2}{\nobreak\noindent\emailaddrname
      \@ifnotempty{##1}{, \ignorespaces##1\unskip}\/:\space
      \ttfamily##2\par}\endgroup}%
  \addresses
  \endgroup
}

\makeatother

\makeatletter
\def\cstar#1{\expandafter\@cstar\csname c@#1\endcsname}
\def\@cstar#1{\ifcase#1\or $\ast$\or $\ast\ast$\or $\ast\ast\ast$\fi}
\AddEnumerateCounter{\cstar}{\@cstar}{$\ast\ast\ast$}
\makeatother

\newlist{conditions}{enumerate}{1}
\newlist{iconditions}{enumerate}{1}
\newlist{inthm}{enumerate}{1}
\setlist[conditions]{label=\normalfont(\alph*),ref=\normalfont(\alph*)}
\setlist[iconditions]{label=\normalfont(\roman*),ref=\normalfont(\roman*)}
\setlist[inthm]{label=\normalfont(\thetheorem.\arabic*),ref=\normalfont(\thetheorem.\arabic*),wide,labelindent=0pt}

\mathchardef\mhyphen="2D

\newcommand{\B}{\mathbb{B}}
\newcommand{\CB}{\mathbb{C}}
\newcommand{\R}{\mathbb{R}}
\newcommand{\C}{\mathcal{C}}
\newcommand{\FC}{\mathcal{F}}
\newcommand{\GC}{\mathcal{G}}

\newcommand{\OC}{\mathcal{O}}
\newcommand{\PC}{\mathcal{P}}
\newcommand{\SC}{\mathcal{S}}
\newcommand{\ZC}{\mathcal{Z}}

\newcommand{\mf}{\mathfrak{m}}

\newcommand{\Cinfty}{\C^{\infty}}

\DeclareMathOperator{\codim}{codim}

\DeclareMathOperator{\rank}{rank}

\newtheorem{theorem}{Theorem}[section]
\newtheorem{claim}[theorem]{Claim}
\newtheorem{proposition}[theorem]{Proposition}
\newtheorem{lemma}[theorem]{Lemma}
\theoremstyle{definition}
\newtheorem{notation}[theorem]{Notation}
\newtheorem{remark}[theorem]{Remark}
\theoremstyle{remark}
\newtheorem{case}{\indent Case}

\DeclarePairedDelimiter\abs{\lvert}{\rvert}%
\DeclarePairedDelimiter\norm{\lVert}{\rVert}%

\makeatletter
\let\oldabs\abs
\def\abs{\@ifstar{\oldabs}{\oldabs*}}
\let\oldnorm\norm
\def\norm{\@ifstar{\oldnorm}{\oldnorm*}}
\makeatother

\title{{Hartogs-type theorems in real algebraic geometry, I} }
\date{}
\author{Marcin Bilski, Jacek Bochnak \and Wojciech Kucharz}

\address{Marcin Bilski\\Institute of Mathematics\\Faculty of Mathematics and Computer
Science\\Jagiellonian University\\\L{}ojasiewicza 6\\30-348
Krak\'ow\\Poland}
\email{Marcin.Bilski@im.uj.edu.pl}

\address{Jacek Bochnak\\Le Pont de l'\'Etang 8\\1323
Romainm\^otier\\Switzerland}
\email{jack3137@gmail.com}

\address{Wojciech Kucharz\\Institute of Mathematics\\Faculty of Mathematics and Computer
Science\\Jagiellonian University\\\L{}ojasiewicza 6\\30-348
Krak\'ow\\Poland}
\email{Wojciech.Kucharz@im.uj.edu.pl}

\usepackage[pdftex, pdfauthor={M. Bilski, J. Bochnak and W. Kucharz}, pdftitle={\thetitle}]{hyperref}

\begin{document}
\maketitle
\thispagestyle{empty}

\begin{abstract}
Let $f:X\rightarrow\R$ be a function defined on a connected nonsingular real algebraic set $X$ in $\R^n$. We prove that regularity 
of $f$ can be detected on either algebraic curves or surfaces in $X.$  
If $\mathrm{dim}X \geq 2$ and 
$k$ is a positive integer, then $f$ is a regular function whenever the restriction $f|_C$\linebreak is a regular function for every algebraic curve 
$C$ in $X$ that is a $\mathcal{C}^k$ submanifold homeomorphic to the unit
circle and is either nonsingular or has precisely one singularity. Moreover, in the latter case, the singularity of $C$ is equivalent to the plane curve singularity defined by the equation $x^p=y^q$
for some primes $p<q.$ If $\mathrm{dim}X\geq 3,$ then $f$ is a regular function whenever the restriction $f|_S$ is a regular function for every nonsingular algebraic surface $S$ in $X$ that is 
homeomorphic to the unit 2-sphere.  
We also have suitable versions of these results {for $X$ not necessarily connected.} 
\end{abstract}

\section{Introduction and main results}\label{sec:1}

The purpose of the present paper is to prove that regularity of real-valued functions defined on a nonsingular real algebraic set $X$ in $\R^n$ can be detected on mildly singular algebraic curves or nonsingular algebraic surfaces in $X$. In some well-defined sense our results are optimal. All theorems announced in this section are proved in Section~\ref{sec:3}.

We refer to either \cite{bib4} or \cite{Man2020} for the general theory of real algebraic sets, real regular functions, and related topics. Unless explicitly stated otherwise, we consider $\R^n$ and all its subsets endowed with the Euclidean topology induced by the standard norm on $\R^n$.

Let $Z$ be an algebraic set in $\R^n$. 
The \emph{algebraic complexification} 
 of $Z$ is the smallest complex algebraic subset $Z^{\CB}$ 
of $\CB^n$ that contains $Z$ ($\R^n$ is viewed as a subset of $\CB^n$). Note that $Z^{\mathbb{C}}$ is nonsingular at every point of $Z$ if and only if $Z$ is nonsingular (in the sense of \cite{bib4}). Assume that $Z$ is nonsingular. 
For any real analytic map-germ $\varphi:(Z,a)\rightarrow(\R^m,b)$, the \emph{complexification} of $\varphi$ is the uniquely determined complex analytic map-germ $\varphi_{\CB}:(Z^{\CB},a)\rightarrow(\CB^m,b)$ whose restriction to $Z$ is equal to $\varphi.$

Let $X$ be an irreducible nonsingular real algebraic set in $\R^n$ of dimension $m\geq 2,$ and $C$ a real algebraic curve in $X.$ Let $(p,q)$ be a pair of primes, $p<q$,  and $$D_{p,q}:=\{(x,y,z_1,\ldots,z_{m-2})\in\R^m: x^p-y^q=0,z_1=\cdots=z_{m-2}=0\}.$$ We say that $C$ is of \textit{type} $(p,q)$ \textit{at} $a\in C$ if there exists a real analytic diffeomorphism-germ $\sigma:(X,a)\rightarrow(\R^m,0)$ whose complexification $\sigma_{\CB}$ maps the germ of $C^{\CB}$ at $a$ onto the germ of $D_{p,q}^{\CB}$ at $0\in\CB^m.$
We say that $C$ is a $(p,q)$-\textit{curve} if it has a unique singular point  $a$ and it is of type $(p,q)$ at $a.$ If $k$ is a nonnegative integer with $pk<q,$ then $D_{p,q}$ is a $\mathcal{C}^k$
submanifold of $\R^m,$ hence any $(p,q)$-curve in $X$ is a $\mathcal{C}^k$ submanifold.

\subsection{Main results (simplified versions)}\label{subsec:1.1}

The main two results of this paper are stated in a simplified form as Theorems \ref{th-1-1} and~\ref{th-1-2} below.

\begin{theorem}\label{th-1-1}
Let $k$ be a positive integer and let $f \colon X \to \R$ be a function defined on a connected nonsingular algebraic set $X$ in $\R^n$, with $\dim X \geq 2$. Then the following conditions are equivalent:\pagebreak[2]
\begin{conditions}
\item\label{th-1-1-a} $f$ is regular on $X$.

\item\label{th-1-1-b} For every algebraic curve $C$ in $X$ that is homeomorphic to the unit circle and is either nonsingular or a $(p,q)$-curve for some primes $p, q$ with $pk<q$ (so in particular is a $\mathcal{C}^k$ submanifold), the restriction $f|_C$ is a regular function.
\end{conditions}
\end{theorem}

Sharper results are contained in Theorem~\ref{th-1-4} (for $X=\R^n$) and Theorem~\ref{th-1-5} (which encompasses also the 
case of $X$ irreducible but not necessarily connected).  Theorems \ref{th-1-4} 
and \ref{th-1-5} are optimal in the sense made precise in Remark~\ref{rem-1-8}. As indicated in \ref{rem-1-8-1}, a real-valued 
function on $\R^n$, for $n \geq 2$, need not be regular (or even continuous) despite the fact that all its restrictions 
to nonsingular algebraic curves in $\R^n$ are regular. Theorems \ref{th-1-1}, \ref{th-1-4} and \ref{th-1-5} can be viewed 
as a rather surprising continuation of the research project undertaken in \cite{bib16}, whose principal aim has been 
a characterization of continuous real rational functions by their restrictions to algebraic curves or arcs of such curves. 
In Remark~\ref{rem-1-9} we briefly comment on the ({being in preparation}) second part of the present paper.

\begin{theorem}\label{th-1-2}
Let $f \colon X \to \R$ be a function defined on a connected nonsingular algebraic set $X$ in $\R^n$, with $\dim X \geq 3$. Then the following conditions are equivalent:
\begin{conditions}
\item\label{th-1-2-a} $f$ is regular on $X$.

\item\label{th-1-2-b} For every nonsingular algebraic surface $S$ in $X$ that is homeomorphic to the unit $2$-sphere, the restriction $f|_S$ is a regular function.
\end{conditions}
\end{theorem}

This result also has sharper versions, Theorem~\ref{th-1-6} (for $X=\R^n$) and Theorem~\ref{th-1-7} (for $X$ not necessarily 
connected). Theorem~\ref{th-1-7} is a significant improvement upon \cite[Theorem~6.2]{bib16}. Theorems \ref{th-1-2}, \ref{th-1-6} 
and \ref{th-1-7} are algebraic analogs of the results obtained in \cite{bib6} for the real analytic category. However, the 
transition to the algebraic setting is not obvious at all and requires new methods. {The results of the present paper can be interpreted as real 
algebraic variants of the classical Hartogs theorem on separately holomorphic functions of several complex variables \cite{bib10}.}

Besides \cite{bib6, bib22}, the key 
ingredient in our proof of Theorem~\ref{th-1-7} is Theorem~\ref{th-1-5}  engaging $(p,q)$-curves introduced above. Let us briefly discuss the property of these curves, called faithfulness, that plays a decisive role in the proof of Theorem~\ref{th-1-7}.  The concept in question refers in fact to real algebraic sets of any dimension. Given a real algebraic set $Z$ in $\R^n$ and a point $b\in Z$, we regard the germ $Z_b$ of $Z$ 
at $b$ as a real analytic set-germ. The \emph{analytic complexification} of $Z_b$ is the smallest complex analytic set-germ 
at $b \in \CB^n$ that contains~$Z_b$, see \cite[pp.~91, 92]{bib25}. We say that $Z$ is \emph{faithful at $b$} if the analytic 
complexification of $Z_b$ is equal to the germ at $b$ of the algebraic complexification $Z^{\CB}$ of $Z$; we say that $Z$ is 
\emph{faithful} if it is faithful at each of its points. Here ``faithful'' replaces ``quasi-regular'' used in \cite{bib27}. 

{If $Z$
is additionally a topological manifold with all connected components of the same dimension, then $Z$ is faithful at $b$ if and only if the complex analytic
germ of $Z^{\CB}$ at $b$ is irreducible.} Clearly, $Z$ is faithful at each nonsingular point. In general, $Z$ need not be faithful at singular points 
(the irreducible algebraic curve $C$ in $\R^2$ given by the equation $x^4 - 2x^2 y - y^3 = 0$ has the only 
singular point at the origin and is not faithful; moreover, $C$ is a real analytic submanifold of $\R^2$). 

By \cite[Proposition~4]{bib26}, $Z$ is faithful at $b$ if and only if, in the ring of real analytic 
function-germs $(\R^n,b) \to \R$, the ideal of function-germs vanishing on $Z_b$ is generated by the germs 
at $b$ of polynomial functions $\R^n \to \R$ vanishing on $Z$. Thus, $Z$ is nonsingular if and only if $Z$ is faithful 
and a real analytic submanifold of $\R^n$, which is equivalent to being faithful and a $\Cinfty$ submanifold of $\R^n$ (by \cite[Chap.~VI, Proposition~3.11]{bib24}).

Let $C$ be a real algebraic curve in an irreducible nonsingular real algebraic subset $X$ of $\R^n$ with $\mathrm{dim}X=m\geq 2.$ Let $(p,q)$ be a pair of primes, $p<q.$ Then $C$ is a $(p,q)$-{curve} if and only if it has a unique singular point $a$, it is faithful at $a,$ and there exists a real analytic diffeomorphism-germ $\sigma:(X,a)\rightarrow(\R^m,0)$ such that $\sigma(C_a)$ is equal to the germ of $D_{p,q}$ at $0\in\R^m$.

\subsection{Main results (full generality)}\label{subsec:1.2}

Given a real-valued function $\alpha$ on some set $\Omega$, we denote by $\ZC(\alpha)$ the zero set of $\alpha$, that is, $\ZC(\alpha) = \{ x \in \Omega : \alpha(x) = 0 \}$.

It will be convenient to consider regular functions in a more general context than usual. Let $A$ be an arbitrary subset of $\R^n$. A function $f \colon A \to \R$ is said to be \emph{regular at a point $a \in A$} if there exist two polynomial functions $\varphi, \psi \colon \R^n \to \R$ such that $\psi(a) \neq 0$ and $f(x) = \frac{\varphi(x)}{\psi(x)}$ for all $x \in A \setminus \ZC(\psi)$; as expected, $f$ is said to be \emph{regular} on~$A$ if it is regular at every point in~$A$. Actually, assuming that $f$ is regular on $A$, one can find two polynomial functions $\Phi, \Psi \colon \R^n \to \R$ with
\begin{equation*}
    A \subset \R^n \setminus \ZC(\Psi) \quad \text{and} \quad f(x) = \frac{\Phi(x)}{\Psi(x)} \quad \text{for all } x \in A,
\end{equation*}
see the proof of \cite[Proposition~3.2.3]{bib4}.

A function {$g \colon A \to \R$} defined on a subset $A$ of $\R^n$ is said to be \emph{real analytic} if for every point $a \in A$ there exist an open neighborhood $U \subset \R^n$ of $a$ and a real analytic function $G \colon U \to \R$ (in the usual sense) such that $G$ and $g$ agree on $A \cap U$. Obviously, every regular function on $A$ is real analytic.

For integers $n$ and $d$, with $1 \leq d \leq n-1$, an algebraic set $\Sigma$ in $\R^n$ is called a \emph{Euclidean $d$-sphere} if it can be expressed as
\begin{equation*}
    \Sigma = \{x \in \R^n : \norm{x - x_0} = r\} \cap Q,
\end{equation*}
where $Q$ is an affine $(d+1)$-plane in $\R^n$, $x_0\in Q$, and $r>0$. A Euclidean $1$-sphere is also called a \emph{Euclidean circle}.

Next we define certain collections of real algebraic curves and surfaces in order to formulate our results in an optimal way. Let $\Pi$ denote the set of all prime numbers. For any positive integer $k,$ put
$$\Lambda^k:=\{(p,q)\in\Pi\times\Pi:pk<q\}.$$

\begin{notation}\label{not-1-3}
Let $X$ be an irreducible nonsingular algebraic set {in $\R^n,$ for $n\geq 2,$} and let $U$ be a nonempty open subset of $X$.

\begin{inthm}[widest=1.3.3]
\item\label{not-1-3-1} For an integer $d$, with $1 \leq d \leq \dim X -1$, we denote by $\SC_d(U)$ the collection of all $d$-dimensional irreducible nonsingular algebraic sets $S$ in $X$, contained in $U$, satisfying one of the following two conditions:

\begin{iconditions}
\item\label{not-1-3-i} If $U$ is connected, then $S$ is homeomorphic to the unit $d$-sphere.

\item\label{not-1-3-ii} If $U$ is disconnected, then $S$ has at most two connected components, each homeomorphic to the unit $d$-sphere.
\end{iconditions}

Only the cases $d=1$ and $d=2$ will be relevant.

\item\label{not-1-3-2} Given a positive integer $k$, we denote by $\FC^k(U)$ the collection of all $(p,q)$-curves $C$ in $X$ for $(p,q)\in\Lambda^k$ such that  $C$ is  contained in $U$ and is homeomorphic to the unit circle. Recall that such curves are $\mathcal{C}^k$ submanifolds of $U.$

\item\label{not-1-3-3} For any primes $p<q,$ define $$H_{p,q}(x,y):=x^p-y^q,\mbox{ }\mbox{ } E_{p,q}:=\mathcal{Z}(H_{p,q})=\{(x,y)\in\R^2:H_{p,q}(x,y)=0\}$$
and
$$F_{p,q}(x,y):=x^p-y^q+x^{2q}+y^{2q},\mbox{ }\mbox{ }C_{p,q}:=\mathcal{Z}(F_{p,q})=\{(x,y)\in\R^2:F_{p,q}(x,y)=0\}.$$ For $(a,b)\in\R^2,$ the translates $$(a,b)+E_{p,q}\mbox{ and }(a,b)+C_{p,q}$$ are algebraic curves in $\R^2$ defined by the equations $$H_{p,q}(x-a,y-b)=0\mbox{ and }F_{p,q}(x-a,y-b)=0,$$ respectively.

For any affine $2$-plane $Q$ in $\R^n,$ we choose once and for all an affine-linear isomorphism 
$\varphi_Q:Q\rightarrow\R^2.$ If $n=2,$ then $Q=\R^2$ and $\varphi_{\R^2}$ is chosen to be the identity map. Given a positive integer $k,$ we denote by $\mathcal{H}^k(\R^n)$ (resp. $\mathcal{G}^k(\R^n))$ the collection of all algebraic curves $C$ in $\R^n$ for which there exist an affine $2$-plane $Q$ in $\R^n$ containing $C$ and a pair of primes $(p,q)\in\Lambda^k$ such that $\varphi_Q(C)$ is a translate of $E_{p,q}$ (resp. $C_{p,q}$). In particular, $\mathcal{H}^k(\R^2)$ (resp. $\mathcal{G}^k(\R^2)$) is the collection of all translates of the curves $E_{p,q}$ (resp. $C_{p,q}$) for $(p,q)\in\Lambda^k.$ The curves in $\mathcal{H}^k(\R^2)$ and in $\mathcal{G}^k(\R^2)$ are therefore given by explicit and simple polynomial equations. By Lemmas~\ref{andiffeo} 
and \ref{lem-2-2}, $\mathcal{G}^k(\R^n)\subset\mathcal{F}^k(\R^n).$
\end{inthm}
\end{notation}

Here is a characterization of regular functions on $\R^n$, for $n \geq 2$.

\begin{theorem}\label{th-1-4}
Let $k,n$ be two integers, $k \geq 1$, $n\geq 2$. Then, for a function $f \colon \R^n \to \R$, the following conditions are equivalent:

\begin{conditions}
\item\label{th-1-4-a} $f$ is regular on $\R^n$.

\item\label{th-1-4-b} The restriction of $f$ to every irreducible algebraic curve in $\R^n$ is a regular function.

\item\label{th-1-4-c} The restriction of $f$ to every algebraic curve, which is either a Euclidean circle in~$\R^n$ or is in the collection $\GC^k(\R^n)$, is a regular function.

\item\label{th-1-4-d} The restriction of $f$ to every Euclidean circle in $\R^n$ is a regular function, and the restriction of $f$ to every algebraic curve in the collection $\GC^k(\R^n)$ is a real analytic function.

\item\label{th-1-4-e} The restriction of $f$ to every algebraic curve, which is either an affine line parallel to one of the coordinate axes of $\R^n$ or is in the collection $\mathcal{H}^k(\R^n),$  is a regular function.

\item\label{th-1-4-f} The restriction of $f$ to every affine line parallel to one of the coordinate axes of $\R^n$ is a regular function, and the restriction of $f$ to every algebraic curve in the collection $\mathcal{H}^k(\R^n)$ is a real analytic function.

\end{conditions}
\end{theorem}

In the general setting we have the following result.

\begin{theorem}\label{th-1-5}
Let $k$ be a positive integer, $X$ an irreducible nonsingular algebraic set in $\R^n$, with $\dim X \geq 2$, and $U$ a nonempty open subset of $X$. Then, for a function $f \colon U \to \R$, the following conditions are equivalent:

\begin{conditions}
\item\label{th-1-5-a} $f$ is regular on $U$.

\item\label{th-1-5-b} The restriction of $f$ to every irreducible algebraic curve in $X$, contained in $U$, is a regular function.

\item\label{th-1-5-c} The restriction of $f$ to every algebraic curve, which is either in the collection $\SC_1(U)$ or in the collection $\FC^k(U)$, is a regular function.

\item\label{th-1-5-d} The restriction of $f$ to every algebraic curve in the collection $\SC_1(U)$ is a regular function, and the restriction of $f$ to every algebraic curve in the collection $\FC^k(U)$ is a real analytic function.
\end{conditions}
\end{theorem}

By the definition of the collections $\SC_1(U)$ and $\FC^k(U)$, Theorem~\ref{th-1-5} is both more general and stronger than Theorem~\ref{th-1-1}. Taking $U=X$ would not simplify the proof of Theorem~\ref{th-1-5} in a significant way. In Theorems \ref{th-1-4} and \ref{th-1-5}, the implications (a)$\Rightarrow$(b)$\Rightarrow$(c)$\Rightarrow$(d) and (a)$\Rightarrow$(e)$\Rightarrow$(f) are obvious; we prove the implications (d)$\Rightarrow$(a) and (f)$\Rightarrow$(a) in Section~\ref{sec:3}. Let us note that deducing from (d) or (f) continuity of $f$, which is seemingly a much simpler task, is not at all obvious.

Regular functions on $\R^n$, with $n \geq 3$, can also be characterized as follows.

\begin{theorem}\label{th-1-6}
For a function $f \colon \R^n \to \R$, with $n \geq 3$, the following conditions are equivalent:

\begin{conditions}
\item\label{th-1-6-a} $f$ is regular on $\R^n$.

\item\label{th-1-6-b} The restriction of $f$ to every Euclidean $2$-sphere in $\R^n$ is a regular function.
\end{conditions}
\end{theorem}

As an application of Theorem~\ref{th-1-5}, we will obtain the following.

\begin{theorem}\label{th-1-7}
Let $X$ be an irreducible nonsingular algebraic set in $\R^n$, with $\dim X \geq 3$, and let $U$ be a nonempty open subset of $X$. Then, for a function $f \colon U \to \R$, the following conditions are equivalent:

\begin{conditions}
\item\label{th-1-7-a} $f$ is regular on $U$.

\item\label{th-1-7-b} The restriction of $f$ to every algebraic surface in the collection $\SC_2(U)$ is a regular function.
\end{conditions}
\end{theorem}

Next, we give some comments on the assumptions in our theorems.

\begin{remark}\label{rem-1-8}
Our results are optimal in the following sense.

\begin{inthm}[widest=1.8.3]
\item\label{rem-1-8-1} Let $g \colon \R^n \to \R$, for $n \geq 2$, be the function defined by
\begin{equation*}
    g(x_1, \ldots, x_n) = 
    \begin{cases}
    \frac{x_1^8 + x_2(x_1^2 - x_2^3)^2}{x_1^{10} + (x_1^2 - x_2^3)^2 + x_3^2 + \cdots + x_n^2} &\text{for } (x_1, \ldots, x_n) \neq (0,\ldots,0)\\
    0 &\text{for }(x_1, \ldots, x_n) = (0,\ldots,0).
    \end{cases}
\end{equation*}
The restriction of $g$ to every nonsingular algebraic curve in $\R^n$ is a regular function,
and the restriction of $g$ to every $1$-dimensional real analytic submanifold of $\R^n$ is a real analytic function, see \cite[Example~2.3]{bib16} for the case $n=2$. However, $g$ is not regular on $\R^n$ because it is not even locally bounded on the curve defined by $x_1^2 - x_2^3 = 0$, $x_3 = 0, \ldots, x_n=0$. Hence in Theorems \ref{th-1-1}, \ref{th-1-4} and \ref{th-1-5} we have to allow algebraic curves with singularities of some type.

\item\label{rem-1-8-2} The real algebraic curves in Theorem~\ref{th-1-1} and in Theorems \ref{th-1-4} and \ref{th-1-5} (except (b)) are faithful and are $\mathcal{C}^k$ manifolds.
As we have already noted, an algebraic set in $\R^n$ is nonsingular if and only if it is faithful and a $\Cinfty$ manifold. So, according to \ref{rem-1-8-1}, if algebraic curves used for testing regularity of functions are to be faithful, they cannot be $\Cinfty$ manifolds at the same time. Dropping faithfulness is undesirable because precisely that is needed in our proofs of Theorems \ref{th-1-2} and \ref{th-1-7}.

\item\label{rem-1-8-3} Suppose that the set $U$ in Theorems \ref{th-1-5} and \ref{th-1-7} is disconnected. Let $U_0$ be a connected component of $U$, and let $f \colon U \to \R$ be the function defined by $f=0$ on $U_0$ and $f=1$ on $U \setminus U_0$. Obviously, $f$ is not a regular function on $U$, but the restriction of $f$ to every connected algebraic set in $X$, contained in $U$, is a constant (hence regular) function. Therefore it is essential that the algebraic sets in the collections $\SC_d(U)$, with $d=1$ and $d=2$, are not necessarily connected.
\end{inthm}
\end{remark}

\begin{remark}\label{rem-1-9}
In the second part of the present paper, currently in preparation, we prove 
that it is 
sufficient to restrict attention to algebraic curves which are analytic manifolds (that admit singularities and therefore cannot be faithful). More precisely, among the results we obtain 
the following:
A~function $f \colon X \to \R$ defined on a connected nonsingular real algebraic set $X$ in $\R^n$ (with $\dim X \geq 2$) is 
regular if and only if for every algebraic curve $C$ in $X$, which has at most one singular point and is a real analytic 
submanifold homeomorphic to the unit circle, the restriction $f|_C$ is a regular function. 
\end{remark}

Our results fit into the research program in real algebraic geometry focusing on continuous rational functions, regulous functions and piecewise-regular functions \cite{bib2, bib3, bib5, bib13, FMQ21, bib16, bib17, bib18, bib20} (see also the recent surveys \cite{bib19, bib21} and the references therein).

The paper is organized as follows. In Section~\ref{sec:2} we prove, by a rather intricate argument, a criterion for analyticity of some real meromorphic functions on nonsingular real algebraic surfaces. This enables us to apply in a novel way the tools developed in \cite{bib6, bib9, bib16, bib22}, leading to the proofs of Theorems \ref{th-1-4}, \ref{th-1-5}, \ref{th-1-6} and \ref{th-1-7} (hence also Theorems \ref{th-1-1} and \ref{th-1-2}) in Section~\ref{sec:3}.

\section{Real meromorphic functions on algebraic surfaces}\label{sec:2}

The following result will play the key role in Section~\ref{sec:3}.

\begin{proposition}\label{prop-2-1}
Let $k$ be a positive integer, $X$ an irreducible nonsingular algebraic set in $\R^n$, with $\dim X = 2$, and $f \colon U \to \R$ a function defined on an open subset $U$ of $X$. Let $a$ be a point in $U$ and let $g,h \colon U \to \R$ be two real analytic functions such that
\begin{equation*}
    \ZC(h)=\{a\} \quad \text{and} \quad f(b)=\frac{g(b)}{h(b)} \quad \text{for all } b \in U \setminus \{a\}.
\end{equation*}
Assume that for every algebraic curve $C$ in the collection $\FC^k(U)$, with singular point at~$a$, the restriction $f|_C$ is a real analytic function. Then $f$ is real analytic on $U$.
\end{proposition}

The proof of Proposition~\ref{prop-2-1} requires some preparation and will be preceded by several auxiliary results. Suggestions conveyed to us by S. Donaldson allowed us to simplify the original proof. The case $X=\R^2$ will be analyzed first.

\subsection{Real meromorphic functions of two variables}\label{subsec:2.1}

Let $\mathbb{K}$ denote either the field $\mathbb{R}$ of real numbers or the field $\mathbb{C}$ of
complex numbers. Let $\OC_n^{\mathbb{K}}$ denote the ring of all  $\mathbb{K}$-analytic function-germs at the origin in 
$\mathbb{K}^n$. If no confusion is possible, we will make no distinction between function-germs and their representatives. The ring $\OC_n^{\mathbb{K}}$ is  
a local ring whose maximal ideal is denoted by $\mf_{\mathbb{K},n}.$
As usual, given a positive integer $i$ and an ideal $\mf$, we denote by $\mf^i$ the $i$th power of $\mf$.

\begin{lemma}
\label{andiffeo} Let $p<q$ be prime numbers.
Then for every $v\in\mf_{\mathbb{K},2}^{2q}$ there is a $\mathbb{K}$-analytic 
diffeomorphism-germ $G:(\mathbb{K}^2,0)\rightarrow(\mathbb{K}^2,0)$ such that $F_{p,q}, H_{p,q}$ regarded as elements
of $\mf_{\mathbb{K},2}$ satisfy $$F_{p,q}+v=H_{p,q}\circ G.$$ 
\end{lemma}
\begin{proof} Clearly, $F_{p,q}+v=x^p(1+f(x,y))+y^q(-1+g(x,y))$, for some $\mathbb{K}$-analytic function-germs $f, g$
vanishing at the origin. Hence, it is sufficient to take $$G(x,y)=(x\sqrt[p]{1+f(x,y)}, y\sqrt[q]{-1+g(x,y)}).$$
\end{proof}
\begin{lemma}\label{remark-2-4}  Let $p<q$ be prime numbers. Let $B$ be an open disc in $\R^2$ centered at the origin, $\psi_1,\ldots,\psi_m$ real-valued analytic functions on $B$ that vanish at the origin, and $v_r=(\psi_1+\cdots\psi_m)^r$ for $r=1,2,\ldots$. Then there are an open neighborhood $V\subset B$ of the origin, a positive integer $r_0,$ and real analytic maps $H_r:V\rightarrow B$ for $r\geq r_0$ such that     

\begin{iconditions}
\item\label{lem-2-3-i} the sequence  $\{H_r\}_{r\geq r_0}$ converges uniformly to the inclusion map $V\hookrightarrow B$ together with the first order partial derivatives, and  
\item\label{lem-2-3-ii} for each $r\geq r_0$ the map  $H_r:V\rightarrow H_r(V)$ is an analytic diffeomorphism for which $$F_{p,q}+v_r=F_{p,q}\circ H_r\mbox{ on }V.$$  
\end{iconditions}
\end{lemma}
\begin{proof} Since $\psi_i(0,0)=0$ for $i=1,\ldots,m,$ there are real analytic functions $\lambda_i, \mu_i:B\rightarrow\R$ such that $$\psi_i(x,y)=x\lambda_i(x,y)+y\mu_i(x,y)\mbox{ for all }(x,y)\in B.$$
Therefore there is a positive integer $r_1$ such that $v_{r_1}(x,y)=x^pu_{r_1}(x,y)+y^qw_{r_1}(x,y)$, where
$u_{r_1}, w_{r_1}:B\rightarrow\R$ are real analytic functions. For $r>r_1$ define real analytic functions $u_r, w_r:B\rightarrow\R$ by the equations
$u_r=u_{r_1}v_{r-r_1}$ and $w_r=w_{r_1}v_{r-r_1},$ and note that
$$v_r(x,y)=x^pu_r(x,y)+y^qw_r(x,y)\mbox{ for all }(x,y)\in B.$$ 
Since $v_1(0,0)=0$, there is an open disc $B_0$ in $\R^2$ centered at the origin such that its closure is contained in $B$ and the sequences $\{u_r|_{B_0}\},$ $\{w_r|_{B_0}\}$ converge uniformly to $0$ together with the first order partial derivatives. 

Shrinking $B_0$ and choosing a sufficiently large integer $r_0>r_1,$ we see that the formulas
$$G_0(x,y)=(x\sqrt[p]{1+x^{2q-p}}, y\sqrt[q]{-1+y^q})$$
$$G_r(x,y)=(x\sqrt[p]{1+x^{2q-p}+u_r(x,u)}, y\sqrt[q]{-1+y^q+w_r(x,y)})$$ for $r\geq r_0$ define real analytic maps from $B_0$ into $\R^2,$ and $G_0:B_0\rightarrow G_0(B_0)$ is a real analytic diffeomorphism. By construction, the sequence $\{G_r|_{B_0}\}_{r\geq r_0}$ converges uniformly to $G_0|_{B_0}$ together with the first order partial derivatives.
By \cite[Lemma 1.3, p. 36]{bibHi}), increasing $r_0,$ we may assume that $G_r:B_0\rightarrow G_r(B_0)$ is a real analytic diffeomorphism for $r\geq r_0.$ Now, let $V\subset\R^2$ be an open neighborhood of the origin with $\bar{V}\subset B_0.$ Increasing $r_0$ once again, we get $G_r(V)\subset G_0(B_0)$ for $r\geq r_0.$ Since on $V,$ we have $F_{p,q}=H_{p,q}\circ G_0$ and
$F_{p,q}+v_r=H_{p,q}\circ G_r,$ for $r\geq r_0,$ the real analytic maps $H_r=(G_0|_{B_0})^{-1}\circ(G_r|_V),$ for $r\geq r_0,$ satisfy the required conditions.

\end{proof}

\begin{lemma}\label{Kirreducible}
Let $p<q$ be prime numbers.
Then for every $v\in\mf_{\mathbb{K},2}^{2q}$ the $\mathbb{K}$-analytic function-germ
$F_{p,q}+v$ is irreducible.
\end{lemma}
\begin{proof} The $\mathbb{K}$-analytic function-germ of $H_{p,q}$ at the origin is irreducible. Hence the assertion is an immediate
consequence of Lemma~\ref{andiffeo}.
\end{proof}

For any germ $A$ of a subset of $\R^2$ at the origin, let $I(A)$ denote the ideal of all function-germs in $\OC^{\R}_2$ vanishing on $A$. 
The following lemma formulated in the real setting clearly has its complex analogue.

\begin{lemma}\label{Rgenerate}
Let $p<q$ be prime numbers. 
Then for every $v\in\mf_{\R,2}^{2q}$ the real analytic set-germ 
$\mathcal{Z}(F_{p,q}+v)$ is irreducible of dimension~$1$, and
\begin{equation*}
    I(\mathcal{Z}(F_{p,q}+v)) = (F_{p,q}+v)\OC^{\R}_2.
\end{equation*}
\end{lemma}

\begin{proof}
The real analytic set-germ $\mathcal{Z}(H_{p,q})$ is irreducible of dimension $1$, and
\begin{equation*}
    I(\mathcal{Z}(H_{p,q})) = (H_{p,q})\OC^{\R}_2.
\end{equation*}
Hence, the assertion follows immediately by Lemma~\ref{andiffeo}.
\end{proof}

\begin{lemma}\label{lem-2-2} Let $k$ be a positive integer and $p, q$ primes with $pk<q.$ 
 Then the following hold:
\begin{iconditions}
\item\label{c-i} $C_{p,q}=\mathrm{graph}(\phi_+)\cup\mathrm{graph}(\phi_-),$ 
where $\phi_{+}, \phi_-:[\alpha,\beta]\rightarrow\R$ are given by
$$\phi_+(x)=\sqrt[q]{\frac{1}{2}+\frac{1}{2}\sqrt{1-4(x^p+x^{2q}})}$$
$$\phi_-(x)=\sqrt[q]{\frac{1}{2}-\frac{1}{2}\sqrt{1-4(x^p+x^{2q}})},$$ and $\alpha<\beta$ are (the only) real roots of $1-4(x^p+x^{2q}).$
\item\label{lem-2-2-ii} The origin is the only singular point of the real algebraic curve $C_{p,q}.$ 
\item\label{lem-2-2-iii} $C_{p,q}$ is a $\mathcal{C}^k$ manifold homeomorphic to the unit circle.
\item\label{lem-2-2-iv} The real algebraic curve $C_{p,q}$ is faithful.
\end{iconditions}
\end{lemma}
\begin{proof}\ref{c-i}. Put $\tilde{F}(x,t):=x^p-t+x^{2q}+t^{2}.$ Solving the equation $\tilde{F}=0$ with respect to $t$
we obtain $t_+, t_-:[\alpha,\beta]\rightarrow\R$ given by
$$t_+(x)={\frac{1}{2}+\frac{1}{2}\sqrt{1-4(x^p+x^{2q}})}$$
$$t_-(x)={\frac{1}{2}-\frac{1}{2}\sqrt{1-4(x^p+x^{2q}})},$$ where $\alpha<\beta$ are real roots of $1-4(x^p+x^{2q}).$
Clearly $\mathcal{Z}(\tilde{F})=\mathrm{graph}(t_+)\cup\mathrm{graph}(t_-).$ Since $F_{p,q}(x,y)=\tilde{F}(x,y^q),$ the claim follows.

\ref{lem-2-2-ii}. Elementary calculation shows that $(0,0)$ is the unique critical point of $F_{p,q}$ lying on  $C_{p,q}$.  

\ref{lem-2-2-iii}. By \ref{lem-2-2-ii}, $C_{p,q}\setminus\{(0,0)\}$ is an analytic manifold. Lemma~\ref{andiffeo} shows that the germ of
$C_{p,q}$ at $(0,0)$ is analytically equivalent to the germ of $\mathcal{Z}(H_{p,q})$ at $(0,0)$. Since $pk<q,$ the curve $\mathcal{Z}(H_{p,q})$ is
a $\mathcal{C}^k$ manifold. Hence, $C_{p,q}$ is also a $\mathcal{C}^k$ manifold. 

By \ref{c-i}, $C_{p,q}$ is homeomorphic to the unit circle.

\ref{lem-2-2-iv} The algebraic complexification $C_{p,q}^{\mathbb{C}}$ of $C_{p,q}$ is given by the complex solutions of the equation $F_{p,q}=0.$ Hence, the complex analytic set-germ of $C_{p,q}^{\mathbb{C}}$ at the origin is irreducible by Lemma~\ref{Kirreducible}. Therefore $C_{p,q}$ is faithful in view of
\ref{lem-2-2-ii} and \ref{lem-2-2-iii}.
\end{proof}

\begin{lemma}\label{lem-2-8}
Let  $\Omega \subset \R^2$ be an open neighborhood of $(0,0) \in \R^2$, and 
$g,h\colon \Omega \to \R$ real analytic functions with $\ZC(h) = \{(0,0)\}$. Assume that the real analytic function
\begin{equation*}
    f \colon \Omega \setminus \{(0,0)\} \to \R, \quad (x,y) \mapsto \frac{g(x,y)}{h(x,y)}
\end{equation*}
does not have a real analytic extension to $\Omega$. Then there is a positive integer ${m_0}$ such that for every real analytic function $F:\Omega\rightarrow\R$ with 
$F\in\mf_{\R,2}^{{m_0}}$ and $(F)\OC^{\R}_2=I(\ZC(F))$     
the restriction $f|_{(\ZC(F)\setminus\{(0,0)\})\cap\Omega}$ does not have 
a real analytic extension to $\ZC(F)\cap\Omega$.
\end{lemma}
\begin{proof} By nonextendability of $f,$ the germ of $g$ at the origin is a nonzero element of $\OC^{\R}_2.$ After shrinking $\Omega$ and dividing $g$ and $h$ by their greatest common divisor we may assume that the germs of $g, h$ are relatively prime in $\OC^{\R}_2$ and, again by nonextendability of $f,$ we still have $h(0,0)=0.$  

Let $U$ be a polydisc about
the origin in $\mathbb{C}^2$ on which the complexification of $h$ is well defined.  
Set $V:=\{(x,y)\in U:h(x,y)=0\}$ and note that the germ of $V$ at the origin is of complex dimension~$1$ as $h(0,0)=0$. Since $g, h$ are relatively prime in $\OC^{\R}_2$, there exists an irreducible complex analytic curve-germ $(T,(0,0))$ contained in $(V,(0,0))$ such that $g$ does not vanish identically 
on $(T,(0,0)).$ Indeed, otherwise $g$ would vanish identically on $(V,(0,0))$. Hence, by the Nullstellensatz, there would be an integer $t$ and a $v\in\OC^{\mathbb{C}}_2$ with $g^t=vh.$ Since $g, h\in\OC^{\R}_2,$ we would have $v\in\OC^{\R}_2,$ contradicting the fact that $g, h$ are relatively prime in $\OC^{\R}_2.$ 

By the Puiseux theorem, we have a holomorphic normalization $\phi:(\mathbb{C},0)\rightarrow (T,(0,0)).$ Then $g\circ\phi:(\mathbb{C},0)\rightarrow(\mathbb{C},0)$
is a non-zero holomorphic function. Define ${m_0}$ to be any integer greater than the order of zero of $g\circ\phi$ at $0$.

Let $F:\Omega\rightarrow\R$ be a real analytic function satisfying the hypotheses of the lemma.
Suppose that $f|_{(\ZC(F)\setminus\{(0,0)\})\cap\Omega}$ has 
a real analytic extension to $\ZC(F)\cap\Omega.$ Then there is a $u\in\OC^{\R}_2$ such that $g-uh\in I(\ZC(F)).$ By assumption
there is a $w\in\OC^{\R}_2$ such that $g-uh=wF.$ Passing to the complexifications of $g, F$ and $w$ in a small neighborhood of the origin we obtain $g\circ\phi=(w\circ\phi)(F\circ\phi)$ in a neighborhood of $0.$
This is a contradiction because the order of zero at $0$ of the right-hand side of the latter equation is greater than that of the left-hand side. 
\end{proof}

\subsection{Generalizing to algebraic surfaces}\label{subsec:2.2}

To begin with we present the following.

\begin{lemma}\label{lem-2-13}
Let $X$ be an irreducible nonsingular real algebraic set in $\R^n$, with $\dim X = m \geq 1$, and let $a$ be a point in $X$. Then there exists a linear map $\pi \colon \R^n \to \R^m$ for which the following hold:
\begin{iconditions}
\item\label{lem-2-13-i} The restriction $\pi|_X \colon X \to \R^m$ is a proper map with finite (some possibly empty) fibers.

\item\label{lem-2-13-ii} The map $\pi|_X$ is transverse to $\pi(a)$.
\end{iconditions}
\end{lemma}

\begin{proof}
{Let $Y \coloneqq X - a$ be the translate of $X$. 
Let $L(n,m)$ be the space of all linear maps from $\R^n$ to $\R^m.$
By a suitable version of the Noether normalization theorem, see \cite[Theorem~13.3]{bib12} and its proof, there is a nonempty 
Zariski open subset $\Omega$ of $L(n,m)$ such that every $\beta \in \Omega$ 
is a surjective map whose restriction $\beta|_Y$ is proper with finite fibers. Moreover, there is $\gamma \in \Omega$} such that
the derivative of $\varphi \coloneqq \gamma|_Y \colon Y \to \R^m$ at the origin $0 \in Y$ is an isomorphism. After a linear coordinate change, we may assume that $\gamma$ 
is the canonical projection $\R^n = \R^m \times \R^{n-m} \to \R^m$.

For any constant $\varepsilon > 0$, we set
\begin{equation*}
    {M_{\varepsilon} \coloneqq \{(t_{ij}) \in\R^{m\times n} : |t_{ij}| < \varepsilon \text{ for } 1 \leq i \leq m, \ 1 \leq j \leq n\}}
\end{equation*}
and consider the map {$\Phi \colon \R^n \times M_{\varepsilon} \to \R^m$} defined by
\begin{equation*}
    \Phi(x,t) = \left(x_1 + \sum_{j=1}^n t_{1j} x_j, \ldots, x_m + \sum_{j=1}^n t_{mj} x_j\right),
\end{equation*}
where {$x = (x_1, \ldots, x_n) \in \R^n$} and $t = (t_{ij}) \in M_{\varepsilon}$. If $\varepsilon$ is sufficiently small, 
{then $\Phi|_{Y\times M_{\varepsilon}}$} is a submersion since for every {point 
$x \in Y\setminus\{0\}$} the restriction of $\Phi$ to $\{x\}\times M_{\varepsilon}$ is a submersion, and $\varphi$ is a submersion at the origin $0 \in Y$. 
Hence, according to the standard consequence of Sard's theorem \cite[p.~48, Theorem~19.1]{bib1}, the map 
$\Phi_t \colon Y \to \R^m$, $\Phi_t(x) = \Phi(x,t),$ is transverse to the origin $0 \in \R^m$ for some $t \in M_{\varepsilon}$. 
{Finally, define $\pi \in L(n,m)$ by $\pi(x)=\Phi(x,t)$ and note that  $\pi|_Y = \Phi_t$.} If $\varepsilon$ is small, then $\pi$ belongs to 
$\Omega$ and has the required properties.
\end{proof}

For the sake of clarity we include the following observation.

\begin{lemma}\label{lem-2-14}
Let $a,b$ be two points in $\R^m$ and let $\varepsilon > 0$ be a constant. For the Euclidean norm on $\R^m$, if $\norm{a}^2 > \varepsilon$ and $\norm{b}^2 < \frac{\varepsilon}{9}$, then $\norm{a-b}^2 > \frac{4\varepsilon}{9}$.
\end{lemma}

\begin{proof}
The conclusion follows immediately from the inequality $\norm{a-b} \geq \norm{a} - \norm{b}$.
\end{proof}

Finally, we are ready to complete the main task of this section.

\begin{proof}[Proof of Proposition~\ref{prop-2-1}]
{Our aim is to prove analyticity of $f$ in some neighborhood of $a$ in $U.$ Since the problem is local,}
we are allowed to shrink $U$ if convenient. We may assume that $k \geq 2$ and the point $a \in U$ is the origin in $\R^n$. By Lemma~\ref{lem-2-13}, there exists a linear map $\pi \colon \R^n \to \R^2$ such that the restriction $\pi|_X \colon X \to \R^2$ is a proper map with finite fibers and is transverse to $(0,0) \in \R^2$. After a linear coordinate change in $\R^n$ we may assume that $\pi \colon \R^n = \R^2 \times \R^{n-2} \to \R^2$ is the canonical projection. Now we choose an open neighborhood $\Omega \subset \R^2$ of $(0,0)$ such that
\begin{equation*}
    (\pi|_X)^{-1}(\Omega) = U_1 \cup \cdots \cup U_s,
\end{equation*}
where the $U_l$ are pairwise disjoint open subsets of $X$, each $\pi|_{U_l} \colon U_l \to \Omega$ is a real analytic diffeomorphism, and $a \in U_1 \subset U$. Let $(x,y,z)$, where $z = (z_1, \ldots, z_{n-2})$, be the variables in $\R^n$. The inverse map $(\pi|_{U_l})^{-1} \colon \Omega \to U_l$ is of the form $(x,y) \mapsto (x,y,\varphi^l(x,y))$, where
\begin{equation*}
    \varphi^l = (\varphi_1^l, \ldots, \varphi_{n-2}^l) \colon \Omega \to \R^{n-2}
\end{equation*}
is a real analytic map. Shrinking $U$ and $\Omega$ we may assume that $U_1 = U$ and, for some open subset $\tilde U$ of $\R^n$ 
and polynomial functions $P_1, \ldots, P_{n-2}$ on $\R^n$, { we have }
\begin{equation}\label{eq-1}
\begin{gathered}
    \det \left(\frac{\partial(P_1, \ldots, P_{n-2})}{\partial(z_1, \ldots, z_{n-2})} (b)\right) \neq 0 \quad \text{for all } b \in U,\\
    U = X \cap \tilde U = \ZC(P_1) \cap \cdots \cap \ZC(P_{n-2}) \cap \tilde U.
\end{gathered}
\end{equation}
By construction,
\begin{equation*}
    \tau \colon \Omega \to U, \quad (x,y) \mapsto (x,y, \varphi^1(x,y))
\end{equation*}
is the inverse of the map $\pi|_U \colon U \to \Omega$. Note that $\tau(0,0)=a,$ where $a$ is the origin in $\R^n$. 

There is an open disc $B\subset\R^2$ centered at $(0,0)$ such that $C_{p,q}\subset B$ for all primes $p<q.$ Rescaling the coordinates in $\R^2$ we may assume that $\Omega$ contains the closure $\bar B$.

{We prove that} the restriction $f|_{U \setminus \{a\}}$ has a real analytic extension to $U$. Suppose that this is not the case. We will construct an algebraic curve
$C\in\mathcal{F}^k(U)$ with a singular point at $a$ such that the restriction $f|_{C\setminus\{a\}}$ does not have a real analytic extension to $C.$ The existence of such a $C$ contradicts the hypothesis.

First note that $f \circ \tau|_{\Omega \setminus \{(0,0)\}}$ does not have a real analytic extension to $\Omega$ and
\begin{equation*}
    (f \circ \tau)(x,y) = \frac{(g \circ \tau)(x,y)}{(h \circ \tau)(x,y)} \quad \text{for all } (x,y) \in \Omega \setminus \{(0,0)\}.
\end{equation*}
Let ${m_0}$ be a positive integer provided by Lemma~\ref{lem-2-8} applied to $g\circ\tau$
and $h\circ\tau.$ Fix a pair $(p,q) \in \Lambda^k$ such that ${m_0}<p.$

The case $s=1$ is easy. Indeed, by Lemma~\ref{Rgenerate} (with $v=0$) and by Lemma~\ref{lem-2-8} the 
restriction $f\circ\tau|_{C_{p,q} \setminus \{(0,0)\}}$ does not have a real analytic extension to $C_{p,q}$. Next, by Lemma~\ref{lem-2-2}, we conclude that the algebraic curve $C = (\pi|_X)^{-1}(C_{p,q})$ is homeomorphic to the unit circle and $a$ is the unique
singular point of $C.$ Moreover, by Lemma~\ref{andiffeo}, the germ $C_a$ is $\R$-analytically equivalent to the germ of $E_{p,q}$ at $(0,0).$ Furthermore, the complex analytic set-germ of algebraic complexification of 
$C$ at $a$ is irreducible because, by Lemma~\ref{Kirreducible}, the complex analytic set-germ of algebraic complexification of 
$C_{p,q}$ at $(0,0)$ is irreducible, hence $C$ is faithful at $a.$ Thus $C$ is an element of the collection $\FC^k(U)$, with singular point at $a$, such that the restriction 
$f|_{C \setminus \{a\}}$ does not have a real analytic extension to $C$.

Henceforth we assume that $s \geq 2$. Choose a constant $\varepsilon > 0$ satisfying
\begin{equation}\label{eq-2}
    \inf_{2 \leq l \leq s} \inf_{(x,y) \in \bar B}
    \sum_{i=1}^{n-2} (\varphi_i^1(x,y) - \varphi_i^l(x,y))^2 > \varepsilon.
\end{equation}
By the Weierstrass approximation theorem, for $i=1,\ldots,n-2$, there exist polynomial functions $\Phi_i \colon \R^2 \to \R$ with $\Phi_i(0,0) = \varphi_i^1(0,0)=0$ and
\begin{equation}\label{eq-3}
    \sup_{(x,y)\in\bar B} \sum_{i=1}^{n-2} (\varphi_i^1(x,y) - \Phi_i(x,y))^2 < \frac{\varepsilon}{9}.
\end{equation}
Using \eqref{eq-2}, \eqref{eq-3} and Lemma~\ref{lem-2-14}, we get
\begin{equation}\label{eq-4}
    \inf_{2 \leq l \leq s} \inf_{(x,y) \in \bar B} \sum_{i=1}^{n-2} (\varphi_i^l(x,y) - \Phi_i(x,y))^2 > \frac{4\varepsilon}{9}.
\end{equation}

Since $C_{p,q}=\mathcal{Z}(F_{p,q})\subset B$ and $F_{p,q}>0$ on the unbounded component of $\R^2\setminus \mathcal{Z}(F_{p,q})$ we have
\begin{equation}\label{eq-5}
    F_{p,q} > 0 \quad \text{on } \R^2 \setminus B.
\end{equation}

Given a positive integer $r$, we define two functions $K_r$, $W_r$ on $\R^n$ by
\begin{align*}
    K_r(x,y,z) &= \left( \sum_{i=1}^{n-2} \frac{3}{\varepsilon} (z_i - \Phi_i(x,y))^2 \right)^r,\\
    W_r(x,y,z) &= F_{p,q}(x,y) + K_r(x,y,z).
\end{align*}
The curve $C$ will be of the form $\mathcal{Z}(W_r|_X)$ for $r$ large enough.

In view of \eqref{eq-4}, for $l = 2, \ldots, s$ and $(x,y) \in \bar B$, we have
\begin{equation}\label{eq-6}
    K_r(x,y,\varphi^l(x,y)) = 
    \left(\sum_{i=1}^{n-2} \frac{3}{\varepsilon} (\varphi_i^l(x,y) - \Phi_i(x,y))^2\right)^r 
    > \left(\frac 4 3 \right)^r.
\end{equation}
By \eqref{eq-3}, for $(x,y) \in \bar B$, we obtain
\begin{equation}\label{eq-7}
    K_r(\tau(x,y)) = 
    \left( \sum_{i=1}^{n-2} \frac{3}{\varepsilon} (\varphi_i^1(x,y) - \Phi_i(x,y))^2\right)^r <
    \left(\frac 1 3\right)^r.
\end{equation}
Since $W_r \circ \tau-F_{p,q} = K_r \circ \tau$ on $\Omega$, it readily follows from \eqref{eq-7} that the sequence 
$\{W_r \circ \tau|_{\bar B}\}$ converges to $F_{p,q}|_{\bar B}$ in the $\C^k$ topology. Moreover, the order
of zero of $K_r \circ \tau$ at the origin tends to infinity as $r$ grows. Consequently, by Lemma~\ref{remark-2-4} with $v_r=K_r\circ\tau$, there is a neighborhood 
of the origin in which, for $r$ large enough, $W_r\circ\tau$ is equivalent to $F_{p,q}$ by an analytic diffeomorphism $(\R^2,(0,0))\rightarrow(\R^2,(0,0))$. It follows that, for large $r,$ the gradient of $W_r\circ\tau$ does not
vanish at any point of $\mathcal{Z}(W_r\circ\tau|_{{B}})\setminus\{(0,0)\}$. In view of Lemma~\ref{lem-2-2}\ref{lem-2-2-iii} we conclude that $\mathcal{Z}(W_r\circ\tau|_{{B}})$ is a $\C^k$ submanifold of $B$. We prove
\begin{claim}\label{ClaimC} For $r$ large enough,
$\ZC(W_r \circ \tau|_B)$ is homeomorphic to $C_{p,q}$. 
\end{claim}
\noindent\textit{Proof of Claim~\ref{ClaimC}.} First let us recall that for any one-dimensional $\mathcal{C}^2$ manifold
$M\subset\mathbb{R}^2,$ by a tubular neighborhood of $M$ we mean a pair
$(N,\eta),$ where $N$ is an open neighborhood of $M$ and $\eta:N\rightarrow M$ is a $\mathcal{C}^1$ map such for
every $x\in M$, $\eta^{-1}(x)$ is a segment normal to $M$ at $x$. It is clear that for every $x_0\in M$ there is an open
neighborhood $E$ in $M$ admitting a tubular neighborhood $(N_E,\eta_E)$ in $\mathbb{R}^2$ of constant radius (that
is, there is a $c>0$ such that $\eta_E^{-1}(x)$ is a segment centered at $x$ of length $c$ for every
$x\in E$).

Let us show that $C_{p,q}$ is homeomorphic to $\mathcal{Z}(W_r\circ\tau)$ for $r$ large enough. Recall that the
gradient of $F_{p,q}$ does not vanish at any $b\in C_{p,q}\setminus\{(0,0)\}.$ Hence for every $b\in
C_{p,q}\setminus\{(0,0)\}$ there are an
$\varepsilon_b>0$ and a tubular neighborhood $(N_b,\eta_b)$ of $N_b\cap C_{p,q}\ni b$ of constant radius such that
\vspace*{2mm}\\
(*)\hspace*{5mm}for every $x\in N_b\cap C_{p,q}$ we have $\inf|(F_{p,q}|_{\eta_b^{-1}(x)})'|>\varepsilon_b.$\vspace*{2mm}\\
We may assume that $W_r\circ\tau$ converges to $F_{p,q}$ uniformly together with the first order partial derivatives on
every $N_b.$

By Lemma~\ref{remark-2-4} with $v_r=K_r\circ\tau$, there are an open neighborhood $V$ of the origin and a sequence $H_r:V\rightarrow H_r(V)$ of analytic diffeomorphisms
converging uniformly to the identity on $V$ together with the first order partial derivatives such that $F_{p,q}\circ
H_r=W_r\circ\tau$. We may assume, shrinking $V$ if necessary, that there is a $\mathcal{C}^1$
function $f:V\rightarrow\mathbb{R}$ whose gradient does not vanish at any point of $V$ such that
$\mathcal{Z}(f)=C_{p,q}\cap V.$ Therefore there are an $\varepsilon_{(0,0)}>0$ and a tubular neighborhood
$(N_{(0,0)},\eta_{(0,0)})$ of $N_{(0,0)}\cap C_{p,q}\ni (0,0)$ of constant radius such that\vspace*{2mm}\\
(**)\hspace*{5mm}for every $x\in N_{(0,0)}\cap C_{p,q}$ we have $\inf|(f|_{\eta_{(0,0)}^{-1}(x)})'|>\varepsilon_{(0,0)}.$\vspace*{2mm}

As $C_{p,q}$ is compact, there are $b_1,\ldots,b_s\in C_{p,q}$ such that $C_{p,q}\subset\bigcup_{j=0}^sN_{b_j},$
where $b_0=(0,0).$ Define $f_{r,j}$ on $N_{b_j}$ to be $f\circ H_r$ if $j=0$ and $W_r\circ\tau$ if $j=1,\ldots,s.$
Set $\varepsilon:=\min_{j=0,\ldots,s}\varepsilon_{b_j}.$ Since $f_{r,0}$ converges to $f|_{N_{b_0}}$ and $f_{r,j}$
converges to $F_{p,q}|_{N_{b_j}}$, for $j=1,\ldots,s,$ we conclude that for
$r$ large enough, by (*) and (**), \vspace*{2mm}\\
\hspace*{13mm}for every $j=0,\ldots,s$ and $x\in N_{b_j}\cap C_{p,q}$ we have $\inf|(f_{r,j}|_{\eta_{b_j}^{-1}(x)})'|>\varepsilon/2.$\vspace*{2mm}\\
Consequently, for every $j=0,\ldots,s$ and $x\in N_{b_j}\cap C_{p,q}$ the function $f_{r,j}$ has precisely one
zero on ${\eta_{b_j}^{-1}(x)}.$ On the other hand, $\mathcal{Z}(f_{r,j})=\mathcal{Z}(W_r\circ\tau)\cap N_{b_j},$
for $j=0,\ldots,s.$ Therefore if $(N,\eta)$ is a tubular neighborhood of $C_{p,q}$ with
$N\subset\bigcup_{j=0}^sN_{b_j}$, then for $r$ large enough,
$\eta|_{\mathcal{Z}(W_r\circ\tau)}:\mathcal{Z}(W_r\circ\tau)\rightarrow C_{p,q}$ is injective. Clearly, this map
is also surjective, hence it is a homeomorphism in view of compactness of its domain.\qed\vspace*{3mm}

We proceed with the proof of Proposition~\ref{prop-2-1}. Let $r_0$ be a positive integer such that for $r\geq r_0$ the 
gradient of $W_r\circ\tau$ does not vanish at any point of $\mathcal{Z}(W_r\circ\tau|_B)\setminus\{(0,0)\}$ and 
$\mathcal{Z}(W_r\circ\tau|_B)$ is a $\mathcal{C}^k$ submanifold of $B$ homeomorphic to $C_{p,q}.$
In view of \eqref{eq-5}, we have
\begin{equation*}
    W_r > 0 \quad \text{on } \R^n \setminus \pi^{-1}(B)
\end{equation*}
for all $r$. In particular,
\begin{equation*}
    \ZC(W_r \circ \tau|_B) = \ZC(W_r \circ \tau).
\end{equation*}
Moreover, using \eqref{eq-6} and increasing $r_0$ if necessary, we get
\begin{equation*}
    W_r > 0 \quad \text{on } (\pi|_X)^{-1}(\bar B) \setminus U
\end{equation*}
for $r \geq r_0$. The algebraic curve $C_r := \ZC(W_r|_X)$ {satisfies}
\begin{equation}\label{eq-8}
    C_r \subset U \quad \text{and} \quad \pi(C_r) = (\pi|_U)(C_r) = \ZC(W_r \circ \tau)
\end{equation}
for $r \geq r_0$. Since $\pi|_U \colon U \to \Omega$ is a real analytic diffeomorphism, the set $C_r$ is a $\C^k$ submanifold of $U$, diffeomorphic to $\ZC(W_r \circ \tau)$ 
(hence diffeomorphic to $C_{p,q}$ which in turn is diffeomorphic to the unit circle by Lemma~\ref{lem-2-2}).

By nonvanishing of the gradient of $W_r \circ \tau$ at any point of 
$\ZC(W_r \circ \tau) \setminus \{(0,0)\}$ for $r\geq r_0,$ and by \eqref{eq-1}, \eqref{eq-8} and the fact that $\pi|_U \colon U \to \Omega$ is a real analytic diffeomorphism, we obtain
\begin{equation*}
    \rank \left(\frac{\partial (W_r, P_1, \ldots, P_{n-2})}{\partial (x,y, z_1, \ldots, z_{n-2})} (b)\right) = n-1
\end{equation*}
for all points $b \in C_r \setminus \{a\}$. Consequently, the algebraic curve $C_r$ is nonsingular at every point $b \in C_r \setminus \{a\}$ if $r\geq r_0$.

We may assume that $r_0>2q.$ From now on we work with $r$ satisfying $r \geq r_0$.  Let us show that $C_r$ is an element of $\FC^k(U).$
By Lemma \ref{andiffeo} with $v=K_r\circ\tau,$
there is an analytic diffeomorphism $G:(\R^2,(0,0))\rightarrow(\R^2,(0,0))$ such that
$W_r \circ \tau=H_{p,q}\circ G.$ Consequently, the germ of the algebraic curve $C_r=\mathcal{Z}(W_r|_X)$ at $a$  is $\R$-analytically equivalent to the germ of $E_{p,q}$ at $(0,0).$
In particular, $a$ is a singular point of $C_r.$ Now, it remains to check that the analytic germ  of the algebraic complexification of $C_r$ at $a$ is irreducible, which implies that $C_r$ is faithful at $a$. 

Let $X^{\mathbb{C}}\subset\mathbb{C}^n$ be the algebraic complexification of $X.$ The analytic germ of the complex curve $$S:=\{(x,y,z)\in X^{\mathbb{C}}: W_r(x,y,z)=0 \}$$ at $a$ 
is irreducible. Indeed, $S_a$ is equivalent via the complexification $\tau_{\mathbb{C}}$ of the analytic diffeomorphism $\tau$ to the germ  
of $$T:=\{(x,y)\in\mathbb{C}^2: F_{p,q}(x,y)+K_r(\tau_{\mathbb{C}}(x,y))=0\}$$ at the origin. Hence, $S_a$ is irreducible as $T_{(0,0)}$ is irreducible
by Lemma~\ref{Kirreducible}. Since $C_r\subset S$ we conclude that $C_r$ is an element of $\FC^k(U).$

Finally, let us check that the restriction $f|_{C_r\setminus\{a\}}$ does not have a real analytic extension to $C_r$. First, by Lemma \ref{Rgenerate} with $v=K_r\circ\tau$ we get $I(\mathcal{Z}(W_r\circ\tau))=(W_r\circ\tau)\mathcal{O}^{\R}_2.$ Then, by Lemma~\ref{lem-2-8} applied to $f\circ\tau|_{\Omega\setminus\{(0,0)\}}$ with $F=W_r\circ\tau,$ we obtain that the restriction $f\circ\tau|_{\mathcal{Z}(W_r\circ\tau)\setminus\{(0,0)\}}$
does not have a real analytic extension to $\mathcal{Z}(W_r\circ\tau)$. Consequently, by \eqref{eq-8}, the restriction $f|_{C_r\setminus\{a\}}$ does not have a real analytic extension to $C_r$. We complete the construction of $C$ by  defining $C:=C_r$ for some $r\geq r_0$.

In conclusion, the restriction $f|_{U \setminus \{a\}}$ has a real analytic extension to $U$, say, $F \colon U \to \R$. The final 
task is to show that $f(a) = F(a)$. {To do this, it is sufficient to prove that there exists an algebraic curve $C$
in the collection $\mathcal{F}^k(U)$ with singular point at $a.$ Indeed, for such a curve, the restrictions $f|_C$, $F|_C$ are real analytic functions 
which agree on $C \setminus \{a\}$, hence we get $f(a) = F(a)$, as required. One can obtain $C$ starting with $C_{p,q}$ for arbitrary  
$(p,q) \in \Lambda^k,$ and repeating the arguments used in the construction of the curve $C_r$ above.}

\end{proof}

\section{Proofs of the main theorems}\label{sec:3}

Let $X$ be an irreducible nonsingular algebraic set in $\R^n$ and let $f \colon U \to \R$ be a function defined on a nonempty open subset $U$ of $X$. We say that $f$ admits a \emph{rational representation} if there exist two polynomial functions $G,H \colon \R^n \to \R$ such that
\begin{equation*}
    f(x) = \frac{G(x)}{H(x)} \quad \text{for all } x \in U \setminus \ZC(H)
\end{equation*}
and $H$ is not identically $0$ on $X$ (no restriction on the values of $f$ on the set $U \cap \ZC(H)$ is imposed).

We will repeatedly make use of the following fact.

\begin{lemma}\label{lem-3-1}
With notation as above, if $f$ is a real analytic function that admits a rational representation, then $f$ is a regular function.
\end{lemma}

\begin{proof}
This is the case since for each point $x \in U$, the ring of germs of real analytic functions at $x$ is faithfully flat over the ring of germs of regular functions at $x$ (the latter assertion follows from \cite[Chap.~III, Proposition~4.10]{bib24}).
\end{proof}

\subsection{Restrictions to algebraic curves}\label{subsec:3.1}

First we give a different proof of the result established in \cite[Corollary 5.2]{bib16}.
\begin{theorem} \label{newproof}
Let $f:U\rightarrow\R$ be a function defined on a connected open set $U$ in $\R^n.$ Assume that the restriction of $f$ is a regular function on each open interval contained in $U$ and parallel to one of the coordinate axes. Then there exist two polynomial functions $G, H: \R^n\rightarrow\R$ such that $$\mathrm{codim} P\geq 2\mbox{ and }f(x)=\frac{G(x)}{H(x)}\mbox{ for all }x\in U\setminus P,$$ where $P:=U\cap\mathcal{Z}(H).$  
\end{theorem}
\begin{proof}
By \cite[Proposition 2.4]{bib23}, there is an open cube $W$ contained in $U$ such that the function $f|_W$ is Nash, so analytic, as a function of $n$ variables. Hence, according to \cite[p. 201, Theorem 5]{bib10}, $f|_W$ admits a rational representation. In view of Lemma~\ref{lem-3-1}, $f|_W$ is actually a regular function. Consequently, there exist two polynomial functions $G, H$ on $\R^n$ such that they have no common factor, $W\subset\R^n\setminus\mathcal{Z}(H)$ and $f(x)={G(x)}/{H(x)}$ for all $x\in W.$

Setting $R:={G}/{H}$ and $P:=U\cap\mathcal{Z}(H)$, we claim that $f=R$ on $U\setminus P.$
Obviously, any two points of $U$ can be connected by a continuous arc composed of closed intervals in $U,$ each of which is parallel to one of the coordinate axes. Since the functions $Hf,$ $G$ on $U$ are analytic with respect to each variable separately and are identical on $W,$ it readily follows that $Hf=G$ on $U.$ This proves the claim. 

It remains to prove that $\mathrm{codim} P\geq 2.$ Suppose that $\mathrm{codim} P=1.$ Then there is a nonsingular point $x\in P$ with $G(x)\neq 0$ and $H(x)=0;$ otherwise $G$ and $H$ would have a common factor. Let $I$ be an open interval in $U$ such that $I$ is parallel to one of the coordinate axes, $I$ is not contained in $P,$ and $x\in I.$ Then $f|_I$ is a regular function satisfying 
$$(f|_I)|_{I\setminus P}=f|_{I\setminus P}=R|_{I\setminus P}.$$
We get a contradiction since $R|_{I\setminus P}$ is unbounded near $x.$
\end{proof}

The following result will also be useful.

\begin{theorem}\label{th-3-2}
Let $f \colon \R^n \to \R$, with $n \geq 2$, be a function whose restriction to every Euclidean circle in $\R^n$ is a regular function. Then $f$ admits a rational representation. Furthermore, there exist two polynomial functions $G,H \colon \R^n \to \R$ such that
\begin{equation*}
    \codim \ZC(H) \geq 2 \quad \text{and} \quad f(x) = \frac{G(x)}{H(x)} \quad \text{for all } x \in \R^n \setminus \ZC(H).
\end{equation*}
\end{theorem}

\begin{proof}
To begin with we prove that the restriction of $f$ to some nonempty open subset $U$ of $\R^n$ is a regular function. Let $\B^n$ denote the open unit ball in $\R^n$. Inversion
\begin{equation*}
    {\mu \colon \R^n \setminus \{0\} \to \R^n \setminus \{0\}}, \quad x \mapsto \frac{x}{\norm{x}^2}
\end{equation*}
is a biregular isomorphism. It maps $\B^n \setminus \{0\}$ onto the complement $\R^n \setminus \bar \B^n$ of the closed unit 
ball $\bar \B^n$ and gives a one-to-one correspondence between the Euclidean circles in $\B^n$ passing through the origin and 
the affine lines in $\R^n$ that are disjoint from $\bar\B^n$. Obviously, every affine line $L$ in $\R^n$ that is parallel to one 
of the coordinate axes and has a nonempty intersection with the open cube $T = (2,3)^n$ in $\R^n$ is disjoint from $\bar \B^n$. According to Theorem~\ref{newproof} and the assumption on $f,$ the restriction of $g:=f\circ\mu$ to some nonempty open subset $W$ of $\R^n\setminus\bar\B^n$ is a regular function. Consequently, $f|_U$ is a regular function, where $U:=\mu(W).$ The rest of the proof is quite similar to that of Theorem~\ref{newproof}.

In view of regularity of $f|_U$, there exist two polynomial functions $G,H$ on $\R^n$ such that they have no common factor, $U \subset \R^n \setminus \ZC(H)$, and $f(x) = G(x)/H(x)$ for all $x \in U$. Setting $R = G/H$ and $P = \ZC(H)$, we claim that $f = R$ on $\R^n \setminus P$. To this end, let $x$ be a point in $\R^n \setminus P$ and let $C$ be a Euclidean circle in $\R^n$ such that $x \in C$ and $C \cap U \neq \varnothing$. The functions $f|_{C \setminus P}$ and $R|_{C \setminus P}$ are regular and agree on the nonempty open subset $U \cap C$ of $C$. Hence $f(x) = R(x)$, which proves the claim.

It remains to prove that $\codim P \geq 2$. Suppose that $\codim P = 1.$ {Then there is a point $x\in P$ such that 
$G(x)\neq 0$ and $H(x)=0$; otherwise $G$ and $H$ would have a common factor. Let} 
$E$ be a Euclidean circle in $\R^n$ {not contained in $P$ such that $x\in E$.} 
Then $f|_E$ is a regular function satisfying
\begin{equation*}
    (f|_E)|_{E \setminus P} = f|_{E \setminus P} = R|_{E \setminus P}.
\end{equation*}
{This is a contradiction since $R|_{E\setminus P}$ is unbounded near $x$.}
\end{proof}

\begin{proof}[Proof of Theorem~\ref{th-1-4}]
As already noted in Section~\ref{sec:1}, it is sufficient to prove \ref{th-1-4-d}$\Rightarrow$\ref{th-1-4-a} and \ref{th-1-4-f}$\Rightarrow$\ref{th-1-4-a}. Suppose that \ref{th-1-4-d} (resp. \ref{th-1-4-f}) holds. By Theorem~\ref{th-3-2} (resp. Theorem~\ref{newproof}), the function~$f$ admits a rational representation. Therefore, according to Lemma~\ref{lem-3-1}, it remains to show that $f$ is a real analytic function.

First assume that $n=2$. By Theorem~\ref{th-3-2} (resp.~Theorem~\ref{newproof}), there exist two polynomial functions $G,H\colon \R^2 \to \R$ such that the 
zero set $\ZC(H)$ is finite and $f(x)=G(x)/H(x)$ for all $x \in \R^2 \setminus \ZC(H)$. By Lemma~\ref{lem-2-8}, 
the function $f$ is real analytic in a neighborhood of every point in $\ZC(H)$. Hence $f$ is real analytic on $\R^2$, as required.

In the general case $n \geq 2$, let $Q$ be an affine $2$-plane in $\R^n$. We have just proved that the restriction $f|_Q$ is an analytic function. By \cite[Theorem~1]{bib9} (see also \cite{bib8}), the function~$f$ is real analytic on $\R^n$.
\end{proof}

For the proof of Theorem~\ref{th-1-5} we need the following special case of \cite[Theorem~1.3]{bib22}.

\begin{theorem}\label{th-3-3}
Let $X$ be an irreducible nonsingular algebraic set in $\R^n$, $f \colon U \to \R$ a~function defined on a nonempty open subset $U$ of $X$, and $d$ an integer satisfying $1 \leq\nobreak d \leq \dim X - 1$. Assume that $U$ has finitely many connected components, and the restriction of $f$ to every algebraic set in the collection $\SC_d(U)$ (defined in Notation~\ref{not-1-3}) is a regular function. Then the function $f$ admits a rational representation. Furthermore, there exist two polynomial functions $G,H \colon \R^n \to \R$ such that
\begin{equation*}
    \codim_U P \geq 2 \quad\text{and}\quad f(x) = \frac{G(x)}{H(x)} \quad \text{for all } x \in U \setminus P,
\end{equation*}
where $P\coloneqq U \cap \ZC(H)$.
\end{theorem}

\begin{proof}[Proof of Theorem~\ref{th-1-5}]
As we already know, it is sufficient to prove \ref{th-1-5-d}$\Rightarrow$\ref{th-1-5-a}. Suppose that \ref{th-1-5-d} holds. Our task is to demonstrate regularity of $f$ on $U$. It will be seen that the argument is somewhat shorter if the set $U$ has finitely many connected components. To begin with we show analyticity of $f$ on $U$.

\begin{case}\label{case-1}
First assume that $\dim X = 2$. Let $U_0$ be a connected component of $U$. By Theorem~\ref{th-3-3}, the restriction $f|_{U_0}$ admits a rational representation, hence there exist two polynomial functions $G,H \colon \R^n \to \R$ such that the set $P = U_0 \cap \ZC(H)$ is finite and the restrictions $g = G|_{U_0}$, $h = H|_{U_0}$ are real analytic functions on $U_0$ with
\begin{equation*}
    \ZC(h) = P \quad \text{and} \quad f(x) = \frac{g(x)}{h(x)} \quad \text{for all } x \in U_0 \setminus P.
\end{equation*}
By Proposition~\ref{prop-2-1}, the function $f$ is real analytic in a neighborhood of every point in~$P$. Consequently, $f|_{U_0}$ is a real analytic function. Therefore $f$ is real analytic on $U$, the connected component $U_0$ being arbitrary.
\end{case}

\begin{case}\label{case-2}
Now assume that $\dim X = m \geq 3$. Pick a point $a$ in $U$. By Lemma~\ref{lem-2-13}, there exists a linear map $\pi \colon \R^n \to \R^m$ for which the following hold:
\begin{iconditions}
\item the restriction $\pi|_X \colon X \to \R^m$ is a proper map with finite (some possibly empty) fibers;

\item the map $\pi|_X$ is transverse to $\pi(a)$.
\end{iconditions}
We can choose a constant $r > 0$ such that
\begin{equation*}
    (\pi|_X)^{-1} (B(\pi(a),r)) = U_1 \cup \cdots \cup U_l,
\end{equation*}
where $B(\pi(a), r) \subset \R^m$ is the open ball centered at $\pi(a)$ with radius $r$, the $U_i$ are pairwise disjoint open 
subsets of $X$, $\pi|_{U_i} \colon U_i \to B(\pi(a), r)$ is a real analytic diffeomorphism {for $i=1,\ldots,l$}, and $a \in U_1 \subset U$. 
Define the map $\rho \colon X \to \R^m$ by
\begin{equation*}
    \rho(x) = \frac 1 r (\pi(x) - \pi(a)) \quad \text{for all } x \in X,
\end{equation*}
and let $\B^m$ denote the open unit ball in $\R^m$. Then
\begin{equation*}
    \rho^{-1}(\B^m) = U_1 \cup \cdots \cup U_l
\end{equation*}
and the restriction $\rho|_{U_i} \colon U_i \to \B^m$ is a real analytic diffeomorphism for $i = 1, 
\ldots, l$. 

Now, let $\Sigma$ be a Euclidean $2$-sphere in $\B^m$ passing through the origin, that is, $\Sigma$ is a subset of~$\B^m$ of the form
\begin{equation*}
    \Sigma = \{y \in \R^m : \norm{y - y_0} = \norm{y_0}\} \cap V,
\end{equation*}
{where $y_0 \in V$,} $0 < \norm{y_0} < \frac 1 2$, and $V$ is a vector subspace of $\R^m$ of dimension $3$. By construction, the preimage $S = \rho^{-1}(\Sigma)$ is a nonsingular algebraic surface in $X$ with $a \in S$. Let $Y$ be the irreducible component of $S$ that contains $a$. By Case~\ref{case-1}, the restriction $f|_{Y \cap U}$ is a real analytic function. It follows that the restriction of the function $f \circ (\rho|_{U_1})^{-1} \colon \B^m \to \R$ to $\Sigma$ is a real analytic function. Since $\Sigma$ under consideration is arbitrary, the function $f \circ (\rho|_{U_1})^{-1}$ is real analytic on $\B^m$ by \cite[Theorem~2]{bib6}. Therefore the restriction $f|_{U_1}$ is a real analytic function, so $f$ is real analytic in a neighborhood of~$a$. Consequently, $f$ is real analytic on $U$, the point $a \in U$ being arbitrary.
\end{case}

Having established analyticity of $f$, we complete the proof as follows. Suppose that the set $U$ has finitely many 
connected components. Then, according to Theorem~\ref{th-3-3}, the function $f$ admits a rational representation, and hence it is regular by Lemma~\ref{lem-3-1}.
If $U$ is an arbitrary nonempty open set, then regularity of $f$ on $U$ follows from Lemma~\ref{lem-3-4} below.
\end{proof}

\begin{lemma}\label{lem-3-4}
Let $X$ be an irreducible nonsingular algebraic set in $\R^n$, $f \colon U \to \R$ a real analytic function defined on a nonempty open subset $U$ of $X$, and $d$ an integer satisfying $1 \leq d \leq \dim X - 1$. Assume that the restriction of $f$ to every algebraic set in the collection $\SC_d(U)$ (defined in Notation~\ref{not-1-3}) is a regular function. Then $f$ is a regular function on $U$.
\end{lemma}

\begin{proof}
Since $f$ is a real analytic function, by Lemma~\ref{lem-3-1}, it is sufficient to prove that $f$~admits a rational representation ($U$ may have infinitely many connected components, so Theorem~\ref{th-3-3} is not directly applicable). Suppose that $U$ is disconnected, and let $U_0$ be a connected component of $U$. By Theorem~\ref{th-3-3}, the restriction $f|_{U_0}$ is a regular function, and hence there exist two polynomial functions $\varphi, \psi \colon \R^n \to \R$ such that
\begin{equation*}
    U_0 \subset \R^n \setminus Z(\psi) \quad\text{and}\quad f(x) = \frac{\varphi(x)}{\psi(x)} \quad \text{for all } x \in U_0.
\end{equation*}
It remains to show that
\begin{equation*}
    \psi(x)f(x) = \varphi(x) \quad\text{for all } x\in U.
\end{equation*}
Suppose this is not the case. Then the set
\begin{equation*}
    W = \{x \in U : \psi(x)f(x) \neq \varphi(x)\}
\end{equation*}
is nonempty and open. Choose a connected component $U_1$ of $U$ so that $U_1 \cap W \neq \varnothing$. Let $D_0, D_1$ be two disjoint closed balls in $\R^d$ and let
\begin{equation*}
    h \colon \Omega = D_0 \cup D_1 \to X
\end{equation*}
be a $\Cinfty$ embedding with $h(D_0) \subset U_0$ and $h(D_1) \subset U_1 \cap W$. Let $\partial \Omega$ be the boundary of~$\Omega$. By \cite[Lemma~2.4]{bib22}, the embedding $h$ can be chosen so that $S \coloneqq h(\partial \Omega)$ belongs to the collection of algebraic sets $\SC_d(U)$. (To invoke \cite[Lemma~2.4]{bib22}, we view $X$ as a Zariski open subset of the set of real points of some nonsingular projective (complex) algebraic variety defined over $\R$, which can be achieved by means of resolution of singularities \cite{bib14} or \cite{bib15}.) Note that $S \cap U_0 \neq \varnothing$ and $S \cap W \neq \varnothing$. By assumption, the restriction $f|_S$ is a regular function. Since the regular functions $(\psi f)|_S$ and $\varphi|_S$ agree on the nonempty open subset $S \cap U_0$ of $S$, we get $(\psi f)|_S = \varphi|_S$, the algebraic set $S$ being irreducible. This gives a contradiction because $S \cap W \neq \varnothing$.
\end{proof}

\subsection{Restrictions to algebraic surfaces}\label{subsec:3.2}

Next we investigate regularity of functions by means of restrictions to nonsingular algebraic surfaces.

\begin{proof}[Proof of Theorem~\ref{th-1-6}]
It is sufficient to prove that \ref{th-1-6-b} implies \ref{th-1-6-a}. Suppose that \ref{th-1-6-b}~holds. By \cite[Theorem~2]{bib6}, $f$ is a real analytic function on $\R^n$. Evidently, it follows from~\ref{th-1-6-b} that the restriction of $f$ to every Euclidean circle in $\R^n$ is a regular function, and hence $f$~admits a rational representation in view of Theorem~\ref{th-3-2}. In conclusion, $f$ is regular on~$\R^n$ by Lemma~\ref{lem-3-1}.
\end{proof}

The proof of Theorem~\ref{th-1-7} requires some additional preparation. Let $X$ be an algebraic set in $\R^n$ and let $Y$ be an algebraic subset of $X$. By a polynomial function on $X$ we mean the restriction of a polynomial function on $\R^n$. Let $\PC(X)$ be the ring of all polynomial functions on $X$ and let $I_{\PC(X)}(Y)$ be the ideal of $\PC(X)$ consisting of all polynomial functions vanishing on $Y$. If $X$ is nonsingular, let $\Cinfty(X)$ be the ring of all real-valued $\Cinfty$ functions on $X$ and let $I_{\Cinfty(X)}(Y)$ be the ideal of $\Cinfty(X)$ consisting of all $\Cinfty$ functions vanishing on~$Y$.

The next two lemmas are variants of some results established in \cite[\S1]{bib27}.

\begin{lemma}\label{lem-3-5}
Let $X$ be a nonsingular algebraic set in $\R^n$ and let $Z$ be a faithful algebraic subset of $X$. Then the ideals $I_{\PC(X)}(Z) \Cinfty(X)$ and $I_{\Cinfty(X)}(Z)$ of the ring $\Cinfty(X)$ are equal.
\end{lemma}

\begin{proof}
Let $x \in Z$ and let $Z_x$ be the germ of $Z$ at $x$. Denote by $\OC_x$ (resp. $\Cinfty_x$) the ring of all analytic (resp. $\Cinfty$) function-germs $(\R^n,x) \to \R$. As recalled in Section~\ref{sec:1}, by \cite[Proposition~4]{bib26}, if a function-germ $\varphi \in \OC_x$ vanishes on $Z_x$, then $\varphi \in I_{\PC(\R^n)}(Z)\OC_x$. In particular, it follows that $Z$ regarded as a real analytic set is coherent. Hence, according to \cite[Chap.~VI, Theorem~3.10]{bib24}, if a function-germ $\psi \in \Cinfty_x$ vanishes on $Z_x$, then ${\psi \in I_{\PC(\R^n)}(Z)\Cinfty_x}$. Since every $\Cinfty$ function on $X$ is the restriction of a $\Cinfty$ function on~$\R^n$, using partition of unity, we get $I_{\Cinfty(X)}(Z) \subset I_{\PC(X)}(Z)\Cinfty(X)$. The reversed inclusion is obvious.
\end{proof}

In view of \cite[Chap.~VI, Theorem~3.10]{bib24}, Lemma~\ref{lem-3-5} does not hold if the assumption on faithfulness of $Z$ is omitted.

\begin{lemma}\label{lem-3-6}
Let $X$ be a compact nonsingular algebraic set in $\R^n$ and let $Z$ be a faithful algebraic subset of $X$. Let $\varphi \colon X \to \R$ be a $\Cinfty$ function vanishing on $Z$. Then there exists a polynomial function $\psi \colon X \to \R$, arbitrarily close to $\varphi$ in the $\Cinfty$ topology, such that $\psi$ vanishes on $Z$.
\end{lemma}

\begin{proof}
By Lemma~\ref{lem-3-5}, $\varphi$ can be written as $\varphi = \varphi_1 P_1 + \cdots + \varphi_s P_s$, where the $P_i$ are polynomial functions on $X$ vanishing on $Z$, and the $\varphi_i$ are $\Cinfty$ functions on $X$. By the Weierstrass approximation theorem, each $\varphi_i$ can be approximated in the $\Cinfty$ topology by a polynomial function $\psi_i$ on $X$. Then $\psi = \psi_1 P_1 + \cdots \psi_s P_s$ is a polynomial function on~$X$, close to $\varphi$ in the $\Cinfty$ topology and vanishing on $Z$.
\end{proof}

\begin{proof}[Proof of Theorem~\ref{th-1-7}]
Obviously, \ref{th-1-7-a} implies \ref{th-1-7-b}. Suppose that \ref{th-1-7-b} holds. Our goal is to demonstrate regularity of 
$f$ on $U$. By the resolution of singularities theorem (see \cite{bib14} or~\cite{bib15}) every nonsingular real algebraic set 
is biregularly isomorphic to a Zariski open subset of a compact nonsingular real algebraic set, so we may assume that the real 
algebraic set $X$ is compact. 

Pick a point $a \in U$. First we show real analyticity of $f$ in a neighborhood of $a$ in $U$. Set $m = \dim X$, and let $\rho \colon X \to \R^m$ be the regular map as in Case~\ref{case-2} of the proof of Theorem~\ref{th-1-5}. Recall that
\begin{equation*}
    \rho^{-1}(\B^m) = U_1 \cup \cdots \cup U_l,
\end{equation*}
where $\B^m$ is the open unit ball in $\R^m$, the $U_i$ {are pairwise disjoint} open subsets of $X$, each $\rho|_{U_i} \colon U_i \to \B^m$ is a real analytic diffeomorphism, $a \in U_1 \subset U$, and $\rho(a) = 0$. Let $\Sigma \subset \B^m$ be a Euclidean $2$-sphere passing thorough the origin,
\begin{equation*}
    \Sigma = \{y \in \R^m : \norm{y-y_0} = \norm{y_0}\} \cap V,
\end{equation*}
where $V$ is a vector subspace of $\R^m$ of dimension $3$ and $y_0 \in V$ with $0 < \norm{y_0} < \frac 1 2$. 
The preimage $S = \rho^{-1}(\Sigma)$ is a nonsingular algebraic surface in $X$ passing through $a$. Let $S_a$ be the connected 
component of $S$ containing $a$. By construction, the real analytic diffeomorphism $\rho|_{U_1} \colon U_1 \to \B^m$ 
transforms $S_a$ onto the sphere $\Sigma$. 

Set
\begin{equation*}
    D = (\rho|_{U_1})^{-1} (\{y \in \R^m : \norm{y-y_0} \leq \norm{y_0}\} \cap V).
\end{equation*}
Then $D \subset X$ is a compact $\Cinfty$ submanifold diffeomorphic to the closed ball $\bar\B^3$, with boundary
$\partial D = S_a$. Clearly, the normal bundle to $D$ in $X$ is trivial. Therefore, by \cite[Theorem~1.12]{bib7}, there 
exists a $\Cinfty$ map $\varphi \colon X \to \R^{m-2}$ transverse to the origin $0 \in \R^{m-2}$ and such that 
$S_a = \varphi^{-1}(0)$. 

Now, fix a positive integer $k$, and let $C$ be an algebraic curve in $S$ that belongs either to 
the collection $\SC_1(S_a)$ or the collection $\FC^k(S_a)$ (defined in Notation~\ref{not-1-3}). In particular, $C$ is a 
faithful algebraic subset of $X$. Hence, by Lemma~\ref{lem-3-6}, there exists a polynomial map $\psi \colon X \to \R^{m-2}$, 
arbitrarily close to $\varphi$ in the $\Cinfty$ topology, such that $C \subset Y \coloneqq \psi^{-1}(0)$. 
If $\psi$ is sufficiently close to $\varphi$, then $\psi$ is transverse to $0 \in \R^{m-2}$, so $Y$ is a nonsingular 
algebraic surface in $X$. Furthermore, in view of \cite[p.~51, Theorem~20.2]{bib1}, the submanifolds 
$\varphi^{-1}(0)$ and $\psi^{-1}(0)$ are diffeomorphic, and hence $Y$ is diffeomorphic to the unit $2$-sphere. 
By \ref{th-1-7-b}, $f|_Y$ is a regular function, which in turn implies that $f|_C$ is a regular function. 
Consequently, by Theorem~\ref{th-1-5} (with $X$ {equal to the irreducible component of $S$ containing $S_a$} and $U=S_a$),  $f|_{S_a}$ is a regular function. Therefore the restriction of the function $f \circ (\rho|_{U_1})^{-1} \colon \B^m \to \R$ to $\Sigma$ is a real analytic function. Since $\Sigma$ is arbitrary, $f \circ (\rho|_{U_1})^{-1}$ is a real analytic function by \cite[Theorem~2]{bib6}. Thus $f|_{U_1}$ is real analytic, hence $f$ is real analytic in a neighborhood of $a$. In conclusion, $f$ is real analytic on $U$, the point $a \in U$ being arbitrary.

Suppose that the set $U$ has finitely many connected components. Then, by Theorem~\ref{th-3-3}, the function $f$ admits a rational representation, and hence $f$ is regular, being real analytic.
If $U$ is an arbitrary open set, then regularity of $f$ on $U$ follows from Lemma~\ref{lem-3-4}.
\end{proof}

\phantomsection

\end{document}